\documentclass[smallextended]{svjour3_VERWURSCHTET}       

\usepackage{amsmath}
\usepackage{amsfonts}
\usepackage{amssymb}
\usepackage{color}
\newcommand{\ee}{\mathrm{e}}
\newcommand\rev[1]{{#1}}

\smartqed  

\begin{document}

\title{Non-satisfiability
of a positivity condition for commutator-free exponential integrators
of order higher than four\thanks{This work has been
supported in part by \rev{the Austrian Science Fund (FWF) under grant P30819-N32 and}
the Vienna Science and Technology Fund (WWTF) under grant MA14-002.}
}

\titlerunning{A positivity condition for exponential integrators}        

\author{Harald Hofst{\"a}tter \and Othmar Koch}


\institute{Harald Hofst{\"a}tter \at
              Universit{\"a}t Wien, Institut f{\"u}r Mathematik \\
              Oskar-Morgenstern-Platz 1, A-1090 Wien, Austria\\
              \email{hofi@harald-hofstaetter.at}
           \and
           Othmar Koch \at
              Universit{\"a}t Wien, Institut f{\"u}r Mathematik \\
              Oskar-Morgenstern-Platz 1, A-1090 Wien, Austria\\
              \email{othmar@othmar-koch.org}
}

\date{Received: date / Accepted: date}

\journalname{Numer. Math.}

\sloppy

\maketitle

\begin{abstract}
We consider commutator-free exponential integrators
\rev{as put forward in [Alverman, A., Fehske, H.: High-order commutator-free exponential
  time-propagation of driven quantum systems. J. Comput. Phys. \textbf{230}, 5930--5956 (2011)]}.
For parabolic problems, it is important for the well-definedness that
such an integrator satisfies a positivity condition such that essentially
it only proceeds forward in time. We \rev{prove} that this
requirement implies maximal convergence order of four for real
coefficients, \rev{which has been conjectured earlier by other authors}.
\end{abstract}

\section{Introduction}

Commutator-free exponential integrators are
effective methods for the numerical solution of non-autonomous evolution equations of
the form
\begin{equation}\label{eq:problem}
u'(t)=A(t)u(t),\quad u(t_0)=u_0,\quad t\in[t_0,T],
\end{equation}
where $A(t)\in\mathbb{C}^{d\times d}$,
usually with large dimension $d$, see
\cite{alvfeh11,alvermanetal12,blacastha17,SergioFernandoMPaper1}.
These methods compute approximations $\{u_n\}$  to the solution of (\ref{eq:problem}) on a grid $\{t_n=t_0+n\tau\}$ with step-size $\tau$ \rev{as}
\begin{equation}\label{eq:int_step_general}
u_{n+1} = \ee^{\tau B_J(t_n,\tau)}\cdots\ee^{\tau B_1(t_n,\tau)}u_n,
\qquad n=0,1,\dots.
\end{equation}
Here,
\begin{equation}\label{eq:defB}
\qquad B_j(t_n,\tau)= \sum_{k=1}^{K}a_{j,k}A(t_n+c_{k}\tau),\quad j=1,\dots,J
\end{equation}
with coefficients $a_{j,k}, c_{k}\in\mathbb{R},\ j=1,\dots,J,\ k=1,\dots,K$  chosen in such a way, that the exponential integrator has a certain convergence order $p\geq 1$,
\begin{equation*}
\|u_n-u(t_n)\| = O(\tau^p).
\end{equation*}
Usually the nodes $c_{k}$ are chosen as the nodes
of a quadrature formula of order $p$.
It is assumed that the matrix exponentials applied to some vector, $\ee^{\tau B_j(t_n,\tau)}v$, can be evaluated efficiently.
\rev{For large problems, this is commonly realized by Krylov methods like the Lanczos
iteration, see for instance~\cite{parlig86,saad92s}, or polynomial or rational approximations \cite{molloa03}.}

The numerical solution of large linear systems of the type (\ref{eq:problem}) has been
extensively studied in the literature. Attention has recently
focussed on \emph{commutator-free methods}
(\ref{eq:int_step_general}), see for instance \cite{blamoa05}.
Earlier mathematical work has centered around the construction of commutator-free methods
to supplement classical Magnus integrators based on commutators.
Commutator-free methods are convenient
to evaluate without storing excessive intermediate results, where the optimal
balance between computational effort and accuracy is sought.
Already  in \cite{blamoa05}, the coefficients for high-order commutator-free
methods were derived based on nonlinear optimization of the free parameters in the order
conditions to minimize local error constants.
With this objective, methods of orders 4--8 were constructed in
\cite{alvfeh11}.

\rev{Alternative approaches to the construction of favorable integrators based on the
evaluation of exponentials rely on the Magnus expansion \cite{isenor99,magnus54}.
In \cite{blanesetal08b} the algebraic framework underlying a systematic construction
of classical Magnus integrators is discussed.
}
\rev{Yet another interesting approach was applied to the Schr{\"o}dinger equation in \cite{kropi14,baderetal17},
where all the computations are performed in the Lie algebra. This
leads to the derivation of exponential integrators in  \cite{baderetal17}.
Unconventional schemes similar to \eqref{eq:int_step_general} but containing in addition
commutators have been proposed in \cite{blacastha17}.}
\rev{Note that commutator-free methods were already considered
in the context of Lie group methods for nonlinear problems on manifolds for example in
\cite{celledoni05,celmarowr03,owren06}. Our analysis applies to the class of methods
(\ref{eq:int_step_general}), a possible extension to other Lie group methods
is not considered here, but may be a topic for future research.}

For $A(t)$ with purely imaginary eigenvalues it does not matter whether some
of the coefficients $a_{j,k}$ \rev{in \eqref{eq:defB}} are negative. This holds
for $A(t)$ anti-hermitian (i.e., for $\frac{1}{\mathrm{i}}A(t)$
hermitian), which usually is the case if (\ref{eq:problem}) is a
(space-discretized) Schr\"odinger equation, see
 \cite{alvfeh11,alvermanetal12}.

On the other hand, for
$A(t)$ with negative real eigenvalues of large modulus, which, e.g.,  is the case if
(\ref{eq:problem}) is a spatially semi-discretized \rev{sectorial operator associated with a} parabolic equation,
poor stability is to be expected if some of the coefficients $a_{j,k}$  are negative.
More specifically, the analysis given in \cite{SergioFernandoMPaper1} shows that
commutator-free exponential integrators applied to
evolution equations of parabolic type are well-defined and stable only if their coefficients
satisfy the {\em positivity condition}
\begin{equation}\label{eq:positivity}
b_{j} = \sum_{k=1}^{K}a_{j,k} >0, \quad j=1,\dots,J.
\end{equation}
In all examples of schemes given in
\cite{alvfeh11,alvermanetal12,blacastha17,SergioFernandoMPaper1},
which involve real coefficients only and are of order higher than four, this condition is not satisfied. In \cite{SergioFernandoMPaper1} it is conjectured that no such schemes exist.
However, no proof is given there. It is the purpose of this paper to give such a proof\footnote{\rev{Note that the
situation is similar to that encountered for \emph{exponential splitting methods}. For this class of
time integrators, no methods of order greater than two exist with only positive real coefficients.
This was first shown in \cite{sheng86}, see also \cite{golkap96}. The ensuing instability can be
avoided by splitting with complex coefficients \cite{Hansen09}, see also \cite{Castella09}.
Splitting methods of high order with complex coefficients have been constructed for example in \cite{BCCM2011}.
Similarly, for exponential commutator-free Magnus-type methods, stable high-order schemes with complex coefficients
have been derived in \cite{blacastha17}.
Moreover, unconventional schemes involving additionally evaluation of some commutators are introduced
there which are stable for parabolic problems.
An error analysis of high-order commutator-free exponential integrators
applied to semi-discretizations of parabolic problems is given in \cite{SergioFernandoMPaper1}.
}}.

\section{Main Result}

\begin{theorem} If $p\geq 5$, then no
commutator-free exponential integrator {\rm (\ref{eq:int_step_general})}, {\rm (\ref{eq:defB})} of
convergence order $p$ involving only real coefficients $a_{j,k}$  exists which satisfies the positivity condition
{\rm (\ref{eq:positivity})}.
\end{theorem}

\begin{proof}
It is sufficient to consider only problems of the special form
\begin{equation}\label{eq:special_problem}
u'(t)=(A_0+tA_1)u(t),\quad u(0)=u_0,
\end{equation}
where $A_0, A_1\in\mathbb{C}^{d\times d}$.
For such problems we have
\begin{equation}\label{eq:int_step}
\quad B_j(t_n, \tau)= b_{j}(A_0+t_nA_1)+\tau y_{j}A_1\quad
\mbox{with}\quad y_{j} = \sum_{k=1}^{K}{\rev{a_{j,k}}}c_k,
\quad j=1,\dots,J,
\end{equation}
and order conditions for the numerical solution to be of convergence order at least $p=5$ are given by
\begin{eqnarray}
&&\sum_{j=1}^{J}b_{j} =1,\label{eq:sum_b}\\
&&\sum_{j=1}^{J}y_{j} =\frac{1}{2},\label{eq:eq_lin_e}\\
&&\sum_{j=1}^{J}\widehat{b}_{j}y_{j} =\frac{1}{3},\nonumber\\
&&\sum_{j=1}^{J}\big(\widehat{b}_{j}^2+\frac{1}{12}b_{j}^2\big)y_{j}=\frac{1}{4},\nonumber\\
&&\sum_{j=1}^{J}\big(\widehat{b}_{j}^3+\frac{1}{4}\widehat{b}_{j}b_{j}^2\big)y_{j}=\frac{1}{5},\label{eq:eq_lin_d}\\
&&\sum_{j=1}^{J}\big(\widehat{y}_{j}^2+\frac{1}{12}y_{j}^2\big)b_{j}=\frac{1}{20},\label{eq:eq_quad}
\end{eqnarray}
where
\begin{equation*}
\widehat{b}_{j} = \sum_{k=1}^{j}b_{k}-\frac{1}{2}b_{j}\quad \mbox{and}\quad
\widehat{y}_{j} = \sum_{k=1}^{j}y_{k}-\frac{1}{2}y_{j},\quad j=1,\dots,J.
\end{equation*}
A proof of these order conditions is given in Appendix~\ref{apdx:proof_order_conditions}.
We will show that the system consisting of equations
(\ref{eq:sum_b}), (\ref{eq:eq_lin_e}), (\ref{eq:eq_lin_d}), (\ref{eq:eq_quad})
has no solution with $b_{j}\in\mathbb{R}_{>0}$ positive and $y_{j}\in\mathbb{R}$ arbitrary ($j=1,\dots,J$).
It is clear that these equations have no solution for $J=1$, we thus
assume $J\geq 2$.
We will treat (\ref{eq:eq_lin_e}) and (\ref{eq:eq_lin_d}) as {linear} equations
and (\ref{eq:eq_quad}) as a {quadratic} equation in the variables
$y_{j}$ and with coefficients depending on the parameters $b_{j}$ subject to the constraint (\ref{eq:sum_b}).

We define vectors
\begin{equation*}
y = (y_1\dots,y_J)^T, \quad
e = (1,\dots,1)^T, \quad
d = \big(\widehat{b}_{1}^3+\frac{1}{4}\widehat{b}_{1}b_{1}^2,
 \dots,\widehat{b}_{J}^3+\frac{1}{4}\widehat{b}_{J}b_{J}^2\big)^T
\end{equation*}
in $\mathbb{R}^{J}$ and matrices
\begin{equation*}
L=\left(
    \begin{array}{ccccc}
    \frac{1}{2}                                    \\
    1  & \frac{1}{2}               &   & \text{\huge0}\\
    1  & 1              & \ddots            \\
    \vdots  & \vdots &  \ddots & \frac{1}{2}            \\
   1   &  1             &  \cdots &  1  & \frac{1}{2}
    \end{array}
\right), \quad
D = \mbox{diag}(b_1,\dots,b_J), \quad
S = L^TDL+\frac{1}{12}D
\end{equation*}
in $\mathbb{R}^{J\times J}$. Then equations  (\ref{eq:eq_lin_e}), (\ref{eq:eq_lin_d}), (\ref{eq:eq_quad})
can be written as
\begin{equation}\label{eq:eqs_y}
e^Ty=\frac{1}{2},\quad d^Ty=\frac{1}{5},\quad
y^TSy=\frac{1}{20},
\end{equation}
respectively.
For the remainder of the proof we will assume that $b_j>0$, $j=1,\dots,J$. We will show that under this assumption the system
of equations (\ref{eq:eqs_y}) has no solution $y\in\mathbb{R}^{J}$.
It is easily seen that for $b_j>0$, $S$ is symmetric positive definite and thus the equation
$y^TSy=\frac{1}{20}$ defines a bounded quadric (a ``hyper-ellipsoid'') in $\mathbb{R}^J$.
Furthermore, a straightforward calculation shows that $\{e,d\}$ is linearly independent for $b_j>0$ and $J\geq 2$,
\rev{since in this case $d_1\neq d_2$}.
From Lemma~\ref{lemma:geometric} in Appendix~\ref{apdx:geometric_lemma} it follows
that the intersection of the two hyperplanes given by the equations
$e^Ty=\frac{1}{2}$ and $d^Ty=\frac{1}{5}$ does not intersect the quadric defined by
$y^TSy=\frac{1}{20}$ if and only if
\begin{equation}\label{eq:cond_with_gamma}
c^T\Gamma^{-1}c>\frac{1}{20},
\end{equation}
where $c=(\frac{1}{2},\frac{1}{5})^T$ and $\Gamma$ denotes the Gram matrix
\begin{equation*}
\Gamma=\left(\begin{array}{cc}e^TS^{-1}e & e^TS^{-1}d \\ e^TS^{-1}d & d^TS^{-1}d \end{array}\right).
\end{equation*}
Condition (\ref{eq:cond_with_gamma}) is equivalent to
\begin{equation}\label{eq:cond1}
(2e-5d)^TS^{-1}(2e-5d)-5\det\Gamma>0,
\end{equation}
which is equivalent to
\begin{equation}\label{eq:cond2}
\frac{\big((e-d)^TS^{-1}(2e-5d)\big)^2+\big(9-5(e-d)^TS^{-1}(e-d)\big)\det\Gamma}{(e-d)^TS^{-1}(e-d)}>0.
\end{equation}
\rev{Indeed the left-hand-sides of (\ref{eq:cond1}) and (\ref{eq:cond2}) are equal, which
is readily verified by a straightforward calculation}.
In (\ref{eq:cond2}) the denominator is positive because $S$ and thus $S^{-1}$ is positive \rev{definite}.
Clearly also the first term of the numerator is positive as is the Gram determinant $\det\Gamma$ for $J\geq 2$.
Note that the positivity of these terms does not depend on the constraint (\ref{eq:sum_b}).

We will now show that
\begin{equation}\label{eq:shitty_ineq}
(\sigma^3e-d)^TS^{-1}(\sigma^3e-d) < \frac{9}{5}\sigma^5,\quad\mbox{where}\quad \sigma=\sum_{j=1}^{J}b_j,
\end{equation}
from which (\ref{eq:cond2}) follows for $\sigma=1$ (i.e., if (\ref{eq:sum_b}) is satisfied),
which will conclude the proof of the theorem.
We proceed using induction on the number $J$ of exponentials.

{\em Base case.} For $J=1$ we have
\begin{equation*}
\sigma=b_1,\quad d=\frac{b_1^3}{4},\quad S= \frac{b_1}{3},\quad S^{-1}=\frac{3}{b_1},
\end{equation*}
such that
\begin{equation*}
(\sigma^3e-d)^TS^{-1}(\sigma^3e-d) \ = \ \left(b_1^3-\frac{b_1^3}{4}\right)^2\cdot\frac{3}{b_1}
\ = \ \frac{27}{16}b_1^5 \ < \ \frac{9}{5}\sigma^5.
\end{equation*}

{\em Inductive step $J\to J+1$.}
Objects with a tilde belong to the $(J+1)$-case, those without a tilde to the
$J$-case such that
\begin{equation*}
\widetilde{\sigma} = \sigma+b_{J+1},\quad
 \widetilde{e} = \left(\begin{array}{c} e\\1\end{array}\right),\quad
 \widetilde{d} = \left(\begin{array}{c} d\\d_{J+1}\end{array}\right), \quad \mbox{etc.}
\end{equation*}
Applying
\begin{equation*}
\left(\begin{array}{cc} A & u \\ u^T & \delta \end{array}\right)^{-1} =
\left(\begin{array}{cc} \displaystyle(A-\frac{1}{\delta}uu^T)^{-1} &\displaystyle \frac{-1}{\delta-u^TA^{-1}u}A^{-1}u
\\[12pt]
\displaystyle\frac{-1}{\delta-u^TA^{-1}u}u^TA^{-1} &\displaystyle\frac{1}{\delta-u^TA^{-1}u} \end{array}\right)
\end{equation*}
(see \cite[eq.~(7.7.5)]{horn85})
and the Sherman-Morrison formula
\begin{equation*}
(A+uv^T)^{-1}=A^{-1}-\frac{1}{1+v^TA^{-1}u}A^{-1}uv^TA^{-1}
\end{equation*}
(see \cite[\S 0.7.4)]{horn85})
to
\begin{equation*}
\widetilde{S} = \left(\begin{array}{cc}
\displaystyle S+b_{J+1}ee^T &\displaystyle\frac{1}{2}b_{J+1}e \\[12pt]
\displaystyle\frac{1}{2}b_{J+1}e^T &\displaystyle \frac{1}{3}b_{J+1}
\end{array}\right)
\end{equation*}
we obtain
\begin{equation*} 
\widetilde{S}^{-1} = \left(\begin{array}{cc} \displaystyle
S^{-1}-\frac{b_{J+1}}{4+b_{J+1}e^TS^{-1}e}S^{-1}ee^TS^{-1} &
\displaystyle\frac{-6}{4+b_{J+1}e^TS^{-1}e}S^{-1}e \\[12pt]
\displaystyle\frac{-6}{4+b_{J+1}e^TS^{-1}e}e^TS^{-1} &
\displaystyle\frac{12}{b_{J+1}}\cdot\frac{1+b_{J+1}e^TS^{-1}e}{4+b_{J+1}e^TS^{-1}e}
\end{array}\right).
\end{equation*}
It follows
\begin{eqnarray}
&&(\widetilde{\sigma}^3\widetilde{e}-\widetilde{d})^T\widetilde{S}^{-1}(\widetilde{\sigma}^3\widetilde{e}-\widetilde{d})\ = \ (\widetilde{\sigma}^3e-d)^TS^{-1}(\widetilde{\sigma}^3e-d) \nonumber\\
&&\qquad+\frac{b_{J+1}}{4+b_{J+1}e^TS^{-1}e}\left(
-\big(e^TS^{-1}(\widetilde{\sigma}^3e-d)\big)^2
-12\frac{\widetilde{\sigma}^3-d_{J+1}}{b_{J+1}}e^TS^{-1}(\widetilde{\sigma}^3e-d)\right.\nonumber\\
&&\qquad\qquad\left.+12\frac{(\widetilde{\sigma}^3-d_{J+1})^2}{b_{J+1}^2}(1+b_{J+1}e^TS^{-1}e) \label{eq:expr3}
\right).
\end{eqnarray}
Substituting
\begin{equation*}
\widetilde{\sigma}^3 
=\sigma^3 +b_{J+1}^3+3b_{J+1}^2(\widetilde{\sigma}-b_{J+1})
+3b_{J+1}(\widetilde{\sigma}-b_{J+1})^2
\end{equation*}
in the first term $(\widetilde{\sigma}^3e-d)^TS^{-1}(\widetilde{\sigma}^3e-d)$ only,
and
\begin{equation*}
d_{J+1} = \widetilde{\sigma}^3-\frac{3}{2}\widetilde{\sigma}^2b_{J+1}+
\widetilde{\sigma}b_{J+1}^2-\frac{1}{4}b_{J+1}^3
\end{equation*}
(which follows by substituting $\widehat{b}_{J+1}=\widetilde{\sigma}-\frac{1}{2}b_{J+1}$ in $d_{J+1}=\widehat{b}^3_{J+1}+\frac{1}{4}\widehat{b}_{J+1}b_{J+1}^2$)
in the other terms
we obtain (to be verified  using a computer algebra system, see
 Appendix~\ref{apdx:maplecheck2})
\begin{eqnarray}
&&(\widetilde{\sigma}^3\widetilde{e}-\widetilde{d})^T\widetilde{S}^{-1}(\widetilde{\sigma}^3\widetilde{e}-\widetilde{d})
\ = \ (\sigma^3e-d)^TS^{-1}(\sigma^3e-d)\label{eq:expr4}\\
&&\quad+\frac{b_{J+1}}{4}\big(
7b_{J+1}^4-36b_{J+1}^3\widetilde{\sigma}+72b_{J+1}^2\widetilde{\sigma}^2-72b_{J+1}\widetilde{\sigma}^3+36\widetilde{\sigma}^4\big)\nonumber\\
&&\quad-\frac{b_{J+1}}{4+b_{J+1}e^TSe}\Big(
\big(\widetilde{\sigma}-b_{J+1})^3e-d\big)^TS^{-1}e
-\frac{1}{2}\big(5b_{J+1}^2-12b_{J+1}\widetilde{\sigma} +6\widetilde{\sigma}^2\big)
\Big)^2. \nonumber
\end{eqnarray}
Disregarding the last term (which is negative) and using the inductive assumption
$(\sigma^3e-d)^TS^{-1}(\sigma^3e-d)<\frac{9}{5}\sigma^5=\frac{9}{5}(\widetilde{\sigma}-b_{J+1})^5$ we obtain
\begin{eqnarray*}
&&(\widetilde{\sigma}^3\widetilde{e}-\widetilde{d})^T\widetilde{S}^{-1}(\widetilde{\sigma}^3\widetilde{e}-\widetilde{d}) \\
&&\qquad< \ \frac{9}{5}(\widetilde{\sigma}-b_{J+1})^5 \\
&&\qquad\qquad
+\frac{b_{J+1}}{4}\big(
7b_{J+1}^4-36b_{J+1}^3\widetilde{\sigma}+72b_{J+1}^2\widetilde{\sigma}^2-72b_{J+1}\widetilde{\sigma}^3+36\widetilde{\sigma}^4\big)\\
&&\qquad= \ \frac{9}{5}\widetilde{\sigma}^5-\frac{1}{2}b_{J+1}^5
\ < \ \frac{9}{5}\widetilde{\sigma}^5,
\end{eqnarray*}
which completes the inductive step for the proof of (\ref{eq:shitty_ineq}). \hfill $\Box$
\end{proof}

\noindent\rev{\emph{Remark:\/} Note that the theorem also covers the situation where (some) $a_{j,k}\in \mathbb{C}$, but $b_j,\, y_j \in \mathbb{R}$.}

\appendix
\section{Proof of the order conditions (\ref{eq:sum_b})--(\ref{eq:eq_quad})}
\label{apdx:proof_order_conditions}
For the global error to have order $p=5$ it is required that the local error
have convergence order $p+1=6$. If, without restriction of generality, we consider only
the first integration step for the special problem (\ref{eq:special_problem}), this condition for the local error is written as
\begin{equation}\label{eq:local_error}
\ee^{\tau b_J A_0+\tau^2 y_J A_1}\cdots\ee^{\tau b_1 A_0+\tau^2 y_1 A_1}u_0
= u(\tau) +O(\tau^6).
\end{equation}
A Taylor expansion of the left-hand side leads to
\rev{(let $\mathbf{k}=(k_1,\dots,k_m), \ |\mathbf{k}| =\sum_{l=1}^m k_l$)}
\begin{equation*}
\ee^{\tau b_J A_0+\tau^2 y_J A_1}\cdots\ee^{\tau b_1 A_0+\tau^2 y_1 A_1}u_0
= \rev{c^{(J)}_\emptyset u_0+}  \hspace*{-10mm}\sum_{\rev{\mathbf{k}=}(k_1,\dots k_m)\atop m\geq \rev{1}, k_l\in\{0,1\}, \rev{|\mathbf{k}|+m}\leq 5 }
\hspace{-12mm}\tau^{\rev{|\mathbf{k}|+m}}
c^{(J)}_{k_1\dots k_m}A_{k_1}\cdots A_{k_m}u_0 + O(\tau^6).
\end{equation*}
Here  for $J=1$ we have \rev{(note that already a subset of
coefficients suffices to derive the order conditions \eqref{eq:sum_b}--\eqref{eq:eq_quad})},
\begin{eqnarray*}
&&c^{(1)}_\emptyset=1,\quad c^{(1)}_{0}=b_1,\quad c^{(1)}_{1}=y_1,\quad
c^{(1)}_{01}=\frac{1}{2}b_1y_1,\quad c^{(1)}_{11}=\frac{1}{2}y_1^2,\\
&&c^{(1)}_{001}=\frac{1}{6}b_1^2y_1,\quad c^{(1)}_{011}=\frac{1}{6}b_1y_1^2,
\quad c^{(1)}_{0001}=\frac{1}{24}b_1^3y_1,
\end{eqnarray*}
and for $J\geq 2$ the coefficients can be computed recursively,
\begin{eqnarray*}
&&c^{(J)}_\emptyset=c^{(J-1)}_\emptyset,\quad
c^{(J)}_{0}=c^{(J-1)}_{0}+ b_Jc^{(J-1)}_\emptyset,\quad
c^{(J)}_{1}=c^{(J-1)}_{1}+ y_Jc^{(J-1)}_\emptyset,\\
&&c^{(J)}_{01}=c^{(J-1)}_{01}+b_Jc^{(J-1)}_{1} +\frac{1}{2}b_Jy_Jc^{(J-1)}_\emptyset,\quad
c^{(J)}_{11}=c^{(J-1)}_{11}+y_Jc^{(J-1)}_{1} +\frac{1}{2}y_J^2c^{(J-1)}_\emptyset,\\
&&c^{(J)}_{001}=c^{(J-1)}_{001}+b_Jc^{(J-1)}_{01}
+\frac{1}{2}b_J^2c^{(J-1)}_{1}+\frac{1}{6}b_J^2y_Jc^{(J-1)}_\emptyset,\\
&&c^{(J)}_{011}=c^{(J-1)}_{011}+b_Jc^{(J-1)}_{11}
+\frac{1}{2}b_Jy_Jc^{(J-1)}_{1}+\frac{1}{6}b_Jy_J^2c^{(J-1)}_\emptyset,\\
&&c^{(J)}_{0001}=c^{(J-1)}_{0001}+b_Jc^{(J-1)}_{001}
+\frac{1}{2}b_J^2c^{(J-1)}_{01}+\frac{1}{6}b_J^3c^{(J-1)}_{1}
+\frac{1}{24}b_J^3y_Jc^{(J-1)}_\emptyset.
\end{eqnarray*}
An inductive argument involving  straightforward but laborious calculations gives
\begin{eqnarray}
&&c^{(J)}_\emptyset=1,\quad c^{(J)}_{0}= \sum_{j=1}^Jb_j,\quad  c^{(J)}_{1} = \sum_{j=1}^Jy_j,
\quad c^{(J)}_{01} = c^{(J)}_{0}c^{(J)}_{1}-\sum_{j=1}^J\widehat{b}_jy_j,
\quad c^{(J)}_{11} = \frac{1}{2}(c^{(J)}_{1})^2,\nonumber\\
&& c^{(J)}_{001}=c^{(J)}_{0}c^{(J)}_{01}-\frac{1}{2}(c^{(J)}_{0})^2c^{(J)}_{1}
-\frac{1}{2}\sum_{j=1}^J\big(\widehat{b}_{j}^2+\frac{1}{12}b_{j}^2\big)y_{j},
\quad c^{(J)}_{011} =\frac{1}{2}\sum_{j=1}^{J}\big(\widehat{y}_{j}^2+\frac{1}{12}y_{j}^2\big)b_{j},\nonumber\\
&& c^{(J)}_{0001} = c^{(J)}_{0}c^{(J)}_{001}- \frac{1}{2}(c^{(J)}_{0})^2c^{(J)}_{01}
+\frac{1}{6}(c^{(J)}_{0})^3c^{(J)}_{1}-\frac{1}{6}\sum_{j=1}^{J}\big(\widehat{b}_{j}^3+\frac{1}{4}\widehat{b}_{j}b_{j}^2\big)y_{j}.\label{eq:coeffs_c}
\end{eqnarray}

Repeated differentiation of the differential equation (\ref{eq:special_problem}) yields
\begin{eqnarray*}
&&u(0)=u_0, \quad u'(0) = A_0u_0, \quad u''(0) = (A_1+A_0^2)u_0,\quad
u'''(0) = (A_0A_1+2A_1A_0+A_0^3)u_0,\\
&&u^{(4)}(0) = (3A_1^2+A_0^2A_1+2A_0A_1A_0+A_1A_0^2+A_0^4)u_0,\\
&&u^{(5)}(0) = (3A_0A_1^2+4A_1A_0A_1+8A_1^2A_0+A_0^3A_1+2A_0^2A_1A_0+3A_0A_1A_0^2
+4A_1A_0^3+A_0^5)u_0.
\end{eqnarray*}
Thus for the Taylor expansion of the right-hand side of (\ref{eq:local_error}) we obtain
\begin{equation*}
u(\tau) = \sum_{q=0}^5\frac{\tau^q}{q!}u^{(q)}(0) + O(\tau^6) =
\rev{s_{\emptyset}u_0 +} \hspace{-7mm}\sum_{\rev{\mathbf{k}=}(k_1,\dots k_m)\atop m\geq \rev{1}, k_l\in\{0,1\}, \rev{|\mathbf{k}|+m} \leq 5 }
\hspace{-9mm}\tau^{\rev{|\mathbf{k}|+m}}
\rev{s}_{k_1\dots k_m}A_{k_1}\cdots A_{k_m}u_0 + O(\tau^6)
\end{equation*}
with coefficients (only those corresponding to the subset of coefficients as in (\ref{eq:coeffs_c}))
\begin{eqnarray}
&&s_{\emptyset}=\frac{1}{0!}=1,\quad s_0=\frac{1}{1!}=1,\quad s_1=\frac{1}{2!}=\frac{1}{2},
\quad s_{01}=\frac{1}{3!}=\frac{1}{6},\quad s_{11}=\frac{3}{4!}=\frac{1}{8},\nonumber\\
&&s_{001}=\frac{1}{4!}=\frac{1}{24},\quad s_{011}=\frac{3}{5!}=\frac{1}{40},
\quad s_{0001}=\frac{1}{5!}=\frac{1}{120}.\label{eq:coeffs_s}
\end{eqnarray}
Equating corresponding coefficients in (\ref{eq:coeffs_c}) and (\ref{eq:coeffs_s})
leads to the order conditions (\ref{eq:sum_b})--(\ref{eq:eq_quad}).

\section{A geometric lemma}\label{apdx:geometric_lemma}
\begin{lemma}\label{lemma:geometric}
Let $\{a_1,\dots,a_m\}$ be a linearly independent set of vectors in $\mathbb{R}^n$ and $S\in\mathbb{R}^{n\times n}$ symmetric positive
definite. Further let $c=(\gamma_1,\dots,\gamma_m)^T\in\mathbb{R}^m$ and $\delta\in\mathbb{R}$.
Then the intersection $\mathcal{I}$ of the $m$ hyperplanes in $\mathbb{R}^n$ given by the equations
$a_1^Tx=\gamma_1,\dots,a_m^Tx=\gamma_m$ intersects the hyper-ellipsoid $\mathcal{Q}$ given by the equation
$x^TSx=\delta$ if and only if it holds
\begin{equation}\label{eq:inter_crit}
c^T\Gamma^{-1}c \leq \delta,
\end{equation}
where $\Gamma = (a_i^TS^{-1}a_j)_{i,j=1}^m$ denotes the Gram matrix of the vectors $a_1,\dots,a_m$ with respect to the scalar product $x^TS^{-1}y$.
\end{lemma}
\begin{proof}
First we consider the special case $S=I_n$ (identity matrix), where
$\mathcal{Q}$ is a hyper-sphere. In  this case $\mathcal{I}$ intersects $\mathcal{Q}$
 if and only if the point $x_*\in\mathcal{I}$ of minimal norm satisfies
\begin{equation}\label{eq:inter_crit1}
 \|x_*\|^2=x_*^Tx_* \leq \delta.
\end{equation}
It is easy to see that this point $x_*$ lies in the linear subspace of $\mathbb{R}^n$ spanned
by $a_1, \dots a_m$ (the normal vectors to the given hyperplanes), i.e., there exists $b=(\beta_1,\dots,\beta_m)^T\in\mathbb{R}^m$ such that
$$ x_* = \beta_1a_1+\dots +\beta_ma_m = Ab,$$
where $A=[a_1\ \cdots \ a_m]\in\mathbb{R}^{n\times m}$. Because $x_*\in\mathcal{I}$ it holds
\begin{equation*}
\Gamma b = A^TAb  = A^Tx_*=c,
\end{equation*}
and thus
\begin{equation*}
 x_*^Tx_* = b^TA^TAb = b^T\Gamma b = c^T\Gamma^{-1}c,
\end{equation*}
which shows that (\ref{eq:inter_crit1}) is equivalent to (\ref{eq:inter_crit}).
This completes the proof for the special case $S=I_n$.

For the general case, the symmetric positive definite matrix $S$ can be written as
\begin{equation*}
  S=X\Lambda X^T
\end{equation*}
with $\Lambda=\mathrm{diag}(\lambda_1,\dots,\lambda_n)$, where  $\lambda_j>0$ are the
eigenvalues of $S$, and $X$ orthogonal.
We define $\widetilde{a}_j=\Lambda^{-1/2}X^Ta_j$, $j=1,\dots,m$.
Then under the transformation of variables $\widetilde x=\Lambda^{1/2}X^Tx$, the equation
$\widetilde{a}_j^T\widetilde{x}=\gamma_j$ is equivalent to $a_j^Tx=\gamma_j$
and $\widetilde{x}^T\widetilde{x}=\delta$ is equivalent to $x^TSx=\delta$.
For these transformed equations the special case from above is applicable.
Using
$\widetilde{A} = [\widetilde{a}_1\ \cdots \ \widetilde{a}_m] = \Lambda^{-1/2}X^TA$
with $A=[a_1\ \cdots \ a_m]$
it follows
that for the transformed equations the corresponding Gram matrix satisfies
$\widetilde{\Gamma}=\widetilde{A}^T\widetilde{A}=A^TX\Lambda^{-1}X^TA=A^TS^{-1}A = \Gamma$
as claimed. \hfill $\Box$
\end{proof}

\section{Maple code for checking (\ref{eq:expr3})$=$(\ref{eq:expr4})}\label{apdx:maplecheck2}
Here, the Maple identifiers \verb|s|, \verb|s1|, \verb|bb|, \verb|dd|
\verb|eSe|, \verb|eSd|, \verb|dSd|
correspond to $\sigma$, $\widetilde{\sigma}$, $b_{J+1}$, $d_{J+1}$, $e^TS^{-1}e$, $e^TS^{-1}d$, $d^TS^{-1}d$, respectively.

\begin{verbatim}
> expr3 := s1^6*eSe-2*s1^3*eSd+dSd+bb*(-(eSe*s1^3-eSd)^2
           -(12*(s1^3-dd))*(eSe*s1^3-eSd)/bb
           +12*(s1^3-dd)^2*(bb*eSe+1)/bb^2)/(bb*eSe+4):
> expr4 := s^6*eSe-2*s^3*eSd+dSd
           +(1/4)*bb*(7*bb^4-36*bb^3*s1+72*bb^2*s1^2
           -72*bb*s1^3+36*s1^4)-bb*((s1-bb)^3*eSe-eSd
           -1/2*(5*bb^2-12*bb*s1+6*s1^2))^2/(bb*eSe+4):
> simplify( subs(dd = s1^3-(3/2)*s1^2*bb+s1*bb^2-(1/4)*bb^3, expr3)
           -subs(s = s1-bb, expr4));
                                      0
\end{verbatim}



\end{document}